\documentclass[12pt]{amsart}
\usepackage{a4wide}

\usepackage{amssymb,amsfonts,amsmath,amsthm}
\usepackage{latexsym}

\usepackage[latin1]{inputenc}
\usepackage{graphicx,epsf}

\usepackage[colorlinks]{hyperref}
\newcommand{\adsurl}[1]{\href{#1}{ADS}}
\providecommand{\url}[1]{\href{#1}{#1}}

\usepackage{color}

\newtheorem{thm}{Theorem}[section]

\newtheorem{rem}[thm]{Remark}
\newtheorem{exmp}[thm]{Example}

\newcommand{\setC}{{\mathord{\mathbb C}}}

\newcommand{\setR}{{\mathord{\mathbb R}}}

\newcommand{\mymarginpar}[1]{}

\newcommand{\C}{{\mathord{\mathcal C}}}

\def\supp{\mathop{\mathrm{supp}}}

\def\part{\mathop{\partial}}

\numberwithin{equation}{section}

\begin{document}
\title[ Asymptotic Properties of Unbounded Quadrature Domains ]
 {Asymptotic Properties of Unbounded Quadrature Domains the 
Plane}

\author[L. Karp]{Lavi Karp*}

\address{%
Department of Mathematics\\ ORT Braude College\\
P.O. Box 78, 21982 Karmiel\\ Israel}

\email{karp@braude.ac.il}

\thanks{*Research  supported  ORT
Braude College's Research Authority}

 

\begin{abstract}
We prove that if $\Omega$ is a simply connected quadrature domain of a 
distribution with compact support and the infinity point belongs  the 
boundary, then the boundary has an asymptotic  curve that is  a straight 
line or a parabola or an infinite ray. In other words, such quadrature 
domains in the plane are perturbations of null quadrature domains.
\end{abstract}

\maketitle

\section{Introduction}

A domain $\Omega$ in the complex plane $ \setC$ is called  a \textit{quadrature 
domain} (QD) of a measure $\mu$ (or a distribution) and for the class of 
analytic 
functions, if $\mu _{\mid_{\setC \setminus \Omega}} = 0$ and
\begin{equation}
\label{eq:q-i}
 \int_\Omega f dA=   \mu(f) \quad \text{for all}\ f\in
AL^1(\Omega).
\end{equation} 
Here $AL^1(\Omega)$ is the space of all analytic and integrable functions in
$\Omega$ and $dA$  the area measure. We may also consider quadrature domains for the class of harmonic functions. In that case the space  $AL^1(\Omega)$ is replaced by $HL^1(\Omega)$, the set of all harmonic and integrable functions in $\Omega$,  and the measure $\mu$ is real. Any QD for the class of harmonic functions is also a QD for the class of analytic functions.  

 A particular class of 
unbounded QDs is the family  of \textit{null quadrature domains}, that is, 
domains $\Omega$ for which
\begin{equation*}
 \int_\Omega f dA=0, \quad\text{for all}\ f\in AL^1(\Omega).
\end{equation*} 
This class comprises  half-planes, the exteriors of parabolas, ellipses and
strips, and the complement of any set  which contains  more than three points 
and lies in a straight line \cite{Sakai_81}.

The main result of the present  paper asserts that if $\Omega$ is a simply 
connected QD of a measure with compact support and with an unbounded boundary, 
then $\Omega$ is asymptotically like a null QD. More precisely, the boundary of 
$\Omega$, $\partial\Omega$, has an asymptotic curve that 
is either a straight line or a parabola or an infinite ray.  

In terms of free boundary problems this result has the following 
interpretation. Let $\chi_\Omega$ denotes the characteristic set of $\Omega$ and assume there is a solution to the overdetermined problem 
\begin{equation*}
 \left\{\begin{array}{ll} \Delta u=\chi_\Omega-g \  \ \text{in}\ & \setR^2\\
 u=|\nabla u|=0 \  \  \text{on}\ & \setR^2\setminus \Omega,
\end{array}\right.
\end{equation*} 
where $g$ has a compact support in $\Omega$, and that the boundary of $\Omega$
is unbounded. Then $\partial\Omega$ has an asymptotic curve as is described above. 

Unbounded QDs in the two dimensional plane were studied by Sakai
\cite{Sakai_82, Sakai_93}, Shapiro \cite{Shapiro_87}, and recently by Lee and Makarov \cite{Lee_Makarov_13}.   Sakai showed   that for a given non--negative measure with compact support there is an
unbounded QD that contains a given null QD \cite[Ch. 11]{Sakai_82}. He used 
variational methods that are  also available in higher dimensions. Shapiro 
proposed the use of an inversion in order to characterize  unbounded
quadrature domains  \cite{Shapiro_87}. This idea was accomplished by Sakai in  \cite{Sakai_93}.  However, the asymptotic behavior of the boundary were not considered in those papers.

Our results are inspired by the study of \textit{contact surfaces}, and in 
particular by the works of Strakhov \cite{Strakhov_74, Strakhov_74_2}, since in 
those works the asymptotic line appears naturally. These problems arise in 
geophysics and in the two dimensional plane they have the following 
formulation. Assume that an unbounded Jordan curve $\Gamma$ separates an 
infinite strip into two domains with different constant 
densities. Assume  also  that the strip is parallel to the $x$--axis. 
Strakhov showed that if $\Gamma$ has an asymptotic line ${\rm Im}(z)=h$ ($z=x+iy$), then the gravitational fields can be computed by the Cauchy integral 
\begin{equation}
\label{eq:Cauchy-integral}
 \int_\Gamma \sigma\frac{\bar w- w+2ih}{w-z}dw,
\end{equation} 
where $\sigma $ is the difference between the two densities and $z$ lies above the strip.
The question is whether the shape of $\Gamma$  can be determined by the Cauchy integral when ${\rm Im}(z)$ is large.   
  
  This type of problem is embedded in the frame of  \textit{inverse problems in 
  potential theory},  in which one aims to determinate the shape of a body from 
  the measurements
  of its Newtonian potential far away from the body itself 
  \cite{Gardiner_Sjodin_08, Isakov_93, Ivanov_56_2, Novikov_1938, Zalcman_87}. 
  Quadrature domains are closely related to these problems 
  \cite{Gustafsson_90,  margulis_95}, for example, in the two dimensional plane 
  the Schwarz function has a significant role in all these types of problems 
  \cite{Aharonov_Shapiro_76, Davis_74, shapiro_92, tsirulskiy_63}. 
 
 An essential tool of the proof is the relation between the conformal mapping 
 in lower half--plane and the Cauchy transform of the measure $\mu$ of the QD 
 (formula (\ref{eq:2.21})).  Conformal mappings have been frequently  
 incorporated in all these types of problems.
 The main feature of these sorts of results is to asserts that a conformal map 
 from the unit disk to a domain $\Omega$ is rational if and only if $\Omega$ is 
 a QD of a  combination of Dirac measures and their derivatives 
 \cite{Aharonov_Shapiro_76, Davis_74, Gustafsson_83}, or equivalently, the 
 complex derivative of the  external logarithmic potential of $\chi_\Omega$ is 
 a rational function \cite{Ivanov_56_2, Strakhov_74_2, tsirulskiy_63}.

 Strakhov used the specific form of the conformal mapping in the lower
 half-plane (see (\ref{eq:6.12}) and (\ref{eq:2.14})) in order to show that 
 contact surfaces are highly non-unique  \cite{Strakhov_74_2}. Moreover, he 
 constructed a  continuous family of third order algebraic curves such that the 
 Cauchy integral (\ref{eq:Cauchy-integral})   has the same value for all the 
 curves in the family when  ${\rm Im}(z)$ is large. In the context of QDs,
 Strakhov's example provides an explicit family of unbounded domains, and each one of them is a QD of the same Dirac measure.
 Furthermore, the domains in the family  converge to a union of a disk and a 
 half--plane, in other words, a union of a disk and a null QD. We will extend 
 Strakhov's example to other types of null QDs, that is, we will construct 
 families of unbounded QDs of a fixed Dirac  measure, and such that their 
 boundary has a parabola, or an infinite ray, as an  asymptotic curve.

 We believe that the structure  of unbounded QDs in higher dimension is similar 
 to the two dimensional. However,   the  corresponding theory 
 in higher dimensions  is very restricted, and most of the problems are open. 
 For example, in \cite{karp_shahgholian_00} it is proved that if 
 $\Omega\subset\setR^n$ is an 
 unbounded QD of a measure with compact support and the complement of $\Omega$ 
 is not too thin at infinity, then inversion of the boundary is a $C^1$ surface 
 near the origin. However, this does not imply that the boundary has an 
 asymptotic plane. For further properties of unbounded QDs in higher dimensions 
 see \cite{karp_margulis_96, sakai_09}.

 The plan of the paper is the following: In the next section we shall first 
 establish the relation between rational conformal mappings in the lower 
 half--plane and quadrature identities. Having established that, the main 
 result follows easily.  Section \ref{sec:contact} deals with contact 
 surfaces. Although Strakhov's model is published in \cite{Strakhov_74}, we 
 shall present its basic idea here, and in particular,  we will emphasize its 
 connections to unbounded QDs. 
 In Section \ref{sec:example} we shall  construct examples of families of 
 unbounded QDs of a fixed Dirac measure, and in each example the boundary  has 
 a different type of an  asymptotic curve.

Throughout  this paper $f^\ast$
stands  for $\bar{f}(\bar z)$ and $\chi_\Omega$ is the characteristic function of a set $\Omega$.
We also may assume that $\Omega$ is the interior of its closure whenever it is
a QD (see cf. \cite[Corollary 2.15]{karp_margulis_96}).

\section{Conformal mappings and quadrature domains }

Conformal mappings have been used extensively in QDs and the inverse problem of 
potential theory. A fundamental property is the relations between the conformal 
mapping,  quadrature identity (\ref{eq:q-i}) and the Cauchy transform. To be 
more specific,  a conformal map 
from the 
unit disk to a domain $\Omega$ is rational if and only if $\Omega$ is a QD of a 
combination of Dirac measures and their derivatives \cite{Aharonov_Shapiro_76, 
Davis_74, Gustafsson_83}, or equivalently, the Cauchy transform  of the measure 
$\chi_\Omega$ is a rational function outside $\Omega$
\cite{Ivanov_56_2, tsirulskiy_63}.

Theorem \ref{thm:2} below comprises similar statements. 
It was proved in the context of contact surfaces under the assumption that the boundary has an asymptotic line 
\cite{fedorova_tsirulskiy, Strakhov_74_2}, and for QDs by Shapiro  under the 
assumption that the Schwarz function $S(z)$ tends to infinity as $z$ goes to 
infinity \cite{Shapiro_87}. Shapiro used that assumption  in order to apply an 
inversion and to  reduce the problem to the one of bounded QDs. This assumption 
was  affirmed later 
by Sakai in \cite{Sakai_93}. Since the proof of Theorem \ref{thm:2} is the core 
for the main result we prove it here. We also slightly extend these previous 
results,  and in addition,  our proof is based  upon the generalized Cauchy 
transform.

 Let $z$ and $\zeta$ be points in the complex plane. For a measure $\mu$, we denote by $\C^\mu$ the Cauchy transform of $\mu$, 
\begin{equation}
\label{eq:Cauchy transform}
 \C^\mu(z):=\dfrac{1}{\pi}\int\dfrac{d\mu(\zeta)}{\zeta-z}.
\end{equation}
 Whenever $\sigma\in L^\infty(\setC)$ and  $\mu=\sigma\chi_DdA$  we will denote the  Cauchy transform by  $C^{\sigma D}$.
 
 The Cauchy transform satisfies the differential 
equation $\partial_{\bar z}\mathcal{C}^\mu=-\mu$ in the distributional sense, 
and it  is well defined whenever $\mu$ has a compact support. But it may not 
converge when the measure has arbitrary support. In order to over come this we 
use the following device that was first  implemented by Bers \cite{Bers_65} and 
later by Sakai \cite{Sakai_81, Sakai_93}.

We modify the Cauchy kernel $(\pi(\zeta-z))^{-1}$ by
\begin{equation}
\label{eq:6.20}
 \mathcal{K}(\zeta,z,a,b)=\dfrac{1}{\pi}\left(\dfrac{1}{\zeta-z}+\dfrac{z-b}{
(b-a)(\zeta-a) }+\dfrac{z-a}{ (a-b)(\zeta-b) } \right), \quad a\neq b.
\end{equation} 
Since $\mathcal{K}(\zeta,z,a,b)=O(|\zeta|^{-3})$ for large $|\zeta|$, the
integral
\begin{equation}
\label{eq:6.21}
 \mathcal{C}^g_{\mathcal{K}}(z)=\int \mathcal{K}(\zeta,z,a,b)g(\zeta)dA(\zeta), 
\end{equation} 
converges for any $g\in L^\infty(\setC)$. The 
transformation (\ref{eq:6.21}) is called a {\it generalized Cauchy transform} 
of $g$.
Obviously if ${a,b}\not\in{\rm supp}(g)$, then
$\partial_{\bar{z}}\mathcal{C}^g_\mathcal{K}=-g$.

Let $\mathbb{H}_{\pm}$ denote the upper/lower half-plane respectively and
consider a conformal map $\psi$, from   $\mathbb{H}_{-}$ onto a domain $\Omega$
of, the form
\begin{equation}
\label{eq:6.12}
z= \psi(w)=q(w)+\varphi(w),
\end{equation} 
where $q$ is a quadratic polynomial,  
\begin{equation}
\label{eq:2.14}
\varphi(w) = \sum_{k=1}^m\sum_{j=0}^{m_k}\dfrac{a_{kj}}{(w-b_k)^{j+1}}
+\sum_{k=1}^n
c_k\int_{d_k}^{d_{k+1}} \dfrac{1}{w-s}ds,
\{b_1,\ldots,b_m, d_1,\ldots,d_{n+1}\}\subset \mathbb{H}_+
\end{equation} 
and $\{a_{kj}\}$, $\{c_k\}$ are complex numbers.
The  integrals are along piecewise  straight lines connecting $d_k$ to 
$d_{k+1}$ and do not pass through the points $b_k$. We assume that the points 
$\{b_k\}_{k=1}^m, \{d_k\}_{k=1}^{n+1}$ are distinct.

For  given points $\{\beta_1,\ldots,\beta_m\}$ and 
$\{\delta_1,\ldots,\delta_{n+1}\}$, we define a distribution $T$ as follows:
\begin{equation}
\label{eq:6.14}
T\left(\phi\right)= \sum_{k=1}^m\sum_{j=0}^{m_k}
\alpha_{kj}\partial_z^j\phi(\beta_k) +\sum_{k=1}^n
\bar{c}_k
\int_{\delta_k}^{\delta_{k+1}}\phi(s)ds,
\end{equation}
where the integrals are along piecewise straight lines  connecting $\delta_k$ 
to $\delta_{k+1}$, $\alpha_{kj}$ and $c_k$ are constants. The test function 
$\phi$ belongs to  $C^\infty(D)$, for any open set $D$ such that $\supp(T)\subset D$.

\begin{thm}
\label{thm:2}
Let $\Omega$ be a simply connected domain with an unbounded boundary
$\partial\Omega$. Then the following are equivalent:
\begin{enumerate}
 \item[(a)] The domain $\Omega$ is QD of a distribution $T$ of the form
(\ref{eq:6.14}), that is,
\begin{equation}
\label{eq:q_i:2}
\int_\Omega f dA=  T(f)
\qquad \text{for all} \ f\in AL^1(\Omega) \tag{\ref{eq:q-i}}.
\end{equation}

\item[(b)]
There is a conformal map $\psi$ from $\mathbb{H}_-$ onto $\Omega$ of the 
form (\ref{eq:6.12})--(\ref{eq:2.14}).

\end{enumerate}
\end{thm}

\begin{rem}
The relations between the coefficients of the conformal mapping $\psi$ and the 
distribution $T$ are as follows: $\beta_k=\psi(\bar{b}_k)$, 
$\delta_k=\psi(\bar{d}_k)$ and  $\{\alpha_{kj}\}$ are depended on $\{a_{kj}\}$ 
through equation  (\ref{eq:2.9}) below.
\end{rem}

\noindent
\textit{Proof}( of Theorem \ref{thm:2}). 
Suppose first that $\Omega=\psi(\mathbb{H}_-)$, where $\psi$ is a conformal mapping of the form (\ref{eq:6.12})--(\ref{eq:2.14}).  
Since  the class consisting  of holomorphic functions $f$ in  a neighborhood of  $\Omega$ such that $ f(z)=O(|z|^{-k})$ at infinity is dense in $AL^1(\Omega)$ for any integer $k\geq 3$ (see 
\cite{hayman_karp_shapiro, Shapiro_87}), 
we may  prove the identity (\ref{eq:6.14}) for $f$ that is holomorphic in a neighborhood  of $\Omega$ and has
arbitrary polynomial decay at infinity.

Let $B_r$ be a ball with radius 
$r$ and center at the origin, and let $\Omega_r$ be the image of 
$\mathbb{H}_-\cap B_r$ under the map $\psi$. Then by Green's theorem,
\begin{equation*}
\begin{split}
   \int\limits_{\Omega_r} f dA &=\frac{1}{2i} \int\limits_{\partial\Omega_r} 
\bar zf(z)dz=\frac{1}{2i}\int\limits_{-r}^r 
\overline{\psi(t)}f(\psi(t))\psi'(t)dt 
 \\ & + \frac{1}{2i}\int\limits_{\{|w|=r, {\rm Im}(w)<0\}} 
\overline{\psi(w)}f(\psi(w))\psi'(w)dw.
\end{split}
\end{equation*}
Noting that $\psi$ has at most a quadratic growth at infinity and $f$ has arbitrary 
polynomial decay, the second integral of the right hand side tends to zero as 
$r$ goes to infinity. Hence
\begin{equation}
\label{eq:2.7}
   \int\limits_{\Omega} f dA =\frac{1}{2i}\int\limits_{-\infty}^\infty
\overline{\psi(t)}f(\psi(t))\psi'(t)dt.
\end{equation}

Let $\overline{\psi(\bar w)}=\psi^\ast(w)$. Since 
$\psi^\ast(t)=\overline{\psi}(t)$ for $t\in\setR$, we may replace 
$\overline{\psi}$ by $\psi^\ast$ in (\ref{eq:2.7}). Note 
$\psi^\ast=q^\ast+\varphi^\ast$ is holomorphic outside  the reflection of the 
singularities of  $\varphi$,  and since $f$ has  polynomial decay at infinity 
we have that  
\begin{equation*}
 \lim_{r\to \infty}\int\limits_{\{|w|=r, {\rm Im}(w)<0\}} 
{\psi^\ast(w)}f(\psi(w))\psi'(w)dw=0.
\end{equation*}
 Hence, we can replace the line 
integral in (\ref{eq:2.7})  by several line integrals 
around the singularities of $\varphi^*$ in $\mathbb{H}_-$.
Let 
$C_\rho(\bar{b}_k)$ be a small circle around $\bar{b}_k$, then 
\begin{equation}
\label{eq:2.9}
\begin{split}
 & \dfrac{1}{2i} \int\limits_{C_\rho(\bar{b}_k)}
\left(q^*(w)+\varphi^*(w)\right)f(\psi(w))\psi'(w) dw \\  = &
\dfrac{1}{2i}\sum_{j=0}^{m_k} \int\limits_{C_\rho(\bar{b}_k)}\left(
\dfrac{\bar{a}_{kj}}{(w-\bar{b}_k)^{j+1}}\right)f(\psi(w))\psi'(w) dw \\ = &
\pi \sum_{j=0}^{m_k} \bar{a}_{kj}{\rm
Res}\left(\dfrac{(f \circ\psi)\psi'}{(w-\bar{b}_{k})^{j+1}}, \bar{b}_k\right)
=: 
\sum_{j=0}^{m_k} \alpha_{kj}\partial_z^{j}f(\beta_k),
\end{split}
\end{equation}
where $\beta_k=\psi(\bar{b}_k)$.
 
Let $\gamma$ be a closed curve in $\mathbb{H}_-$ around the polygonal curve
connecting $\bar{d}_1$ to $\bar{d}_{n+1}$, and such that it does not 
surround any of the points $\bar{b}_i, i=1,\ldots,m$. Then 
\begin{equation*}
\begin{split}
 & \dfrac{1}{2i} \int\limits_{\gamma}
\left(q^*(w)+\varphi^*(w)\right)f(\psi(w))\psi'(w) dw \\  = & \dfrac{1}{2i}
\int\limits_{\gamma}\left(\sum_{k=1}^n
\bar{c}_k\int_{\bar{d}_k}^{\bar{d}_{k+1}}\dfrac{1}{w-s
}ds\right)f(\psi(w))\psi'(w) dw\\
= & \dfrac{1}{2i}
\sum_{k=1}^n
\bar{c}_k\int_{\bar{d}_k}^{\bar{d}_{k+1}}\left(\int\limits_{\gamma}\dfrac{1}{w-s
}f(\psi(w))\psi'(w) dw\right)ds\\
= &\pi\sum_{k=1}^n
\bar{c}_k\int_{\bar{d}_k}^{\bar{d}_{k+1}}f(\psi(s))\psi'(s)  ds
= \pi\sum_{k=1}^n
\bar{c}_k\int_{\delta_k}^{\delta_{k+1}}f(\tau)  d\tau,
\end{split}
\end{equation*} 
where $\delta_k=\psi(\bar{d}_k)$.
Summing all the integrals around the singularities of $\varphi^\ast$ in
$\mathbb{H}_-$, we obtain  the quadrature identity (\ref{eq:q_i:2}) with $T$ 
as in (\ref{eq:6.14}). 

For the  converse assertion  we follow Aharonov and Shapiro method 
\cite{Aharonov_Shapiro_76} with some obvious modifications, and we shall also
use Sakai's regularity results \cite{Sakai_91, Sakai_93} to overcome the 
difficulty near the infinity point.

Calculating the generalized Cauchy transform of the distribution $T$, we find 
that
\begin{equation}
\label{eq:6.22}
 \mathcal{C}_{\mathcal{K}}^T(z)=\dfrac{1}{\pi}\sum_{k=1}^m\sum_{j=0}^{m_k}
\frac{(-1)^jj!\alpha_{kj}}{\left(z-\beta_k\right)^{j+1}}+\dfrac{1}{\pi}\sum_{k=1
} ^n
\bar{c}_k\int_{\delta_k}^{\delta_{k+1}}\dfrac{1}{s-z}ds+L(z),
\end{equation}
where $L(z)$ is a linear function.  Let $a,b\not\in\Omega$.
Then for any $z\not\in \Omega$ the modified Cauchy kernel  
$\mathcal{K}(\cdot,z,a,b)$ belongs to $ AL^1(\Omega)$. Hence, the quadrature
identity (\ref{eq:q-i}) implies that
$\mathcal{C}_{\mathcal{K}}^{\Omega}(z)=\mathcal{C}_{ 
\mathcal{K} } ^T(z)$ for $z\not\in\Omega$. Since  $\mathcal{C}_\mathcal{K}^T(z)$
is analytic for $z$ 
outside  the support of $T$, and $\supp(T)\Subset\Omega$, 
\begin{equation*}
 S(z)=\bar{z}+\mathcal{C}_{\mathcal{K}}^{\Omega}(z)-\mathcal{C}_{\mathcal{K}}
^T(z)
\end{equation*} 
is the  Schwarz function of the boundary $\partial\Omega$ in $\Omega$. That is, 
$S(z)=\bar z$ on $\partial\Omega$ and it is holomorphic for $z\in \Omega$ 
near    $\partial\Omega$.

Let now $\psi$ be a conformal map from $\mathbb{H}_-$ to $\Omega$ and set 
\begin{equation}
\label{eq:6.23}
 F(w)=\left\{\begin{array}{ll} S(\psi(w)), \ &w\in \mathbb{H}_-\\
\psi^*(w), & w\in \mathbb{H}_+ \end{array}\right..
\end{equation} 
Since  $S(\psi(t))=\overline{\psi(t)}=\psi^\ast(t) $ for $t\in \setR$, the function $F$ is 
analytically continuable to a single-valued function across the real line (here
we 
rely of Sakai's regularity result that implies that the boundary is a union of 
piecewise analytic arcs \cite{Sakai_91}). The crucial point is the analytic 
continuation near infinity. In order to overcome this difficulty we use an 
inversion,   we set $\Omega_{\rm inv}=\{z: 1/{z}\in\Omega\}$ and 
$\partial\Omega_{\rm inv}=\{z:1/z\in\partial\Omega\}$. Then  it follows from 
Sakai's regularity result \cite[Corollary 2.6]{Sakai_93}, that
\begin{equation*}
 S_{\rm inv}(z):=\left\{\begin{array}{ll}\frac{1}{S({z}^{-1})}, \ & z\neq0\\
0,\ & z=0\end{array}\right.
\end{equation*} 
is the Schwarz function of $\partial\Omega_{\rm inv}$.
Let $\Psi(\zeta):= 1/\psi(-\zeta^{-1})$ and set
\begin{equation*}
\widetilde{F}(\zeta)=\left\{\begin{array}{ll} S_{\rm inv}(\Psi(\zeta)), \
&\zeta\in
\mathbb{H}_-\\ \Psi^*(\zeta), & \zeta\in \mathbb{H}_+ \end{array}\right..
\end{equation*} 
Then the  function $\Psi(\zeta)$ maps  $\mathbb{H}_-$ conformally onto 
$\Omega_{\rm inv}$ and   $ S_{\rm inv}(\Psi(t))=\overline{\Psi(t)}$ for 
$t$ in  $ \setR$. Hence $\widetilde{F}$ has an analytic continuation across 
the real line, and  in particular,  in a neighborhood of the origin. Thus  $F$ in 
(\ref{eq:6.23}) is holomorphic in a neighborhood of infinity. 

 Let $\{\bar{b}_1,\ldots,\bar{b}_m, \bar{d}_1,\ldots,\bar{d}_{n+1}\}\subset 
\mathbb{H}_-$ be the pre-image of $\{\beta_1,\ldots,\beta_m\}$ and 
$\{\delta_1,\ldots,\delta_{n+1}\}$ under $\psi$.  From (\ref{eq:6.22}) and 
(\ref{eq:6.23}) we see that $F'(w)$ is a single-valued meromorphic  function in 
the Riemann sphere  with poles of order higher than or equal to two at the 
points  $\{\bar{b}_1,\ldots,\bar{b}_m\}$, and of order one at the points 
$\{\bar{d}_1,\ldots,\bar{d}_{n+1}\}$. Since $\psi$ is univalent in 
$\mathbb{H}_-$,   $F$ has  polynomial growth at infinity, and hence $F$ is a
meromorphic function in the entire plane  of the form
\begin{equation}
\label{eq:6.24}
 F(w)=q^\ast(w)+\sum_{k=1}^n\sum_{j=0}^{m_k}
\dfrac{\bar{a}_{kj}}{(w-\bar{b}_k)^{j+1}}+\sum_{k=1}^{n+1}
\gamma_k\log(w-\bar{d}_k),
\end{equation} 
where $q$ is a polynomial and $a_{kj},\gamma_k$ are constants. 
From (\ref{eq:6.23}) we see that $F$ is a
single--valued function and hence $\sum_{k=1}^{n+1}
\gamma_k\log(w-\bar{d}_k)$  is also a single-valued. This implies that
$\sum_{k=1}^{n+1}\gamma_k=0$, and
\begin{equation}
\label{eq:2.13}
 \sum_{k=1}^{n+1}
\gamma_k\log(w-\bar{d}_k)=\sum_{k=1}^{n}
\bar{c}_k\int_{\bar{d}_k}^{\bar{d}_{k+1}}\dfrac{1}{w-s}ds
\end{equation} 
for some constants $\{c_k\}$ (see Remark \ref{remark:1} below). Setting
\begin{equation*}
 \varphi^*(w)=\sum_{k=1}^m\sum_{j=0}^{m_k}
\dfrac{\bar{a}_{kj}}{(w-\bar{b}_k)^{j+1}}+\sum_{k=1}^{n}
\bar{c}_k\int_{\bar{d}_k}^{\bar{d}_{k+1}}\dfrac{1}{w-s}ds,
\end{equation*}
 we see  from (\ref{eq:6.23}), (\ref{eq:6.24}) and
(\ref{eq:2.13})  that
\begin{equation*}
 \psi(w)=q(w)+\varphi(w),
\end{equation*} 
where $\varphi$ is given by (\ref{eq:2.14}).

It remains to prove that  the degree of $q$ does not exceed two. Let $N$ be the 
degree of $q$, then we may write $\psi(w)=Aw^Ng(w)$, where $g$ is 
holomorphic in $\mathbb{H}_-$ and $g(\infty)=1$. Therefore   $\Psi(\zeta):= 
1/\psi(-\zeta^{-1})={(-\zeta)^N}{\left(A g(-\zeta^{-1})\right)}^{-1}$. Now, 
according to Sakai's regularity result \cite{Sakai_91, Sakai_93}, there are 
three possibilities: Either $\partial\Omega_{\rm inv}$ is a smooth analytic 
curve near 
the origin, or $\partial\Omega_{\rm inv}$ is a union of two tangential  analytic 
arcs in a neighborhood of the origin,  or $\partial\Omega_{\rm inv}$ has cusp 
singularity at the origin. In the first and second cases, $\Psi(t)$ is a smooth 
curve for $t\in (-\epsilon,\epsilon)$ and  some positive $\epsilon$. Then 
clearly  $N=1$. In the third one $\Psi$ has a Taylor expansion  
$\kappa\left(\zeta^2+a_3 \zeta^3+\cdots\right)$,  for $\zeta\in\mathbb{H}_-\cap 
\{|\zeta|<\epsilon\}$ and $\kappa\in\setC$. Then   obviously $N=2$. So we 
conclude that 
\begin{equation}
\label{eq:2.15}
 q(w)=A_2w^2+A_1w+A_0
\end{equation}
for some complex numbers $A_2,A_1$ and $A_0$,  and this completes the proof. 
 \hfill{$\square$}

The boundary of $\Omega$, $\partial\Omega$, has a parametric representation
$\{\psi(t)=q(t)+\varphi(t): t\in\setR\}$. Since   $\lim_{t\to\pm\infty}\varphi(t)=0$,
$\partial\Omega$ has the asymptotic of the curve $\{q(t): t\in \setR\}$. 
Hence there are three possibilities of asymptotic curves: In case $A_2=0$ in (\ref{eq:2.15}), then
$\partial\Omega$ has an asymptotic of a straight line. Or else, we may write
\begin{equation*}
 q(w)=A_2\left(w^2+\frac{A_1}{A_2}w+\frac{A_0}{A_2}\right).
\end{equation*}
So if ${\rm Im}({A_1}/{A_2})\neq 0$, then $\partial\Omega$ has an asymptotic of
a parabola, otherwise it has an
asymptotic of a ray.

This phenomenon of the asymptotic behavior can be   extended for
 unbounded QDs  for any  distribution with compact support.

\begin{thm}
 Let $\Omega$ be a simply connected quadrature domain of a distribution $T$ 
with compact support in $\Omega$. If $\infty\in\partial\Omega$, then the 
boundary $\partial \Omega$   has  an asymptotic curve that is either a straight 
line or a parabola or an infinite ray. 
\end{thm}

\begin{proof}
 We may assume that the distribution $T$ can be represented by  a smooth 
function $\rho$ with compact support in $\Omega$, that is,
$T(\phi)=\int \phi(z)\rho(z)dA$ (see e.g.
\cite[Lemma 4.3]{shapiro_92}). Since 
$\mathcal{C}_{\mathcal{K}}^\Omega(z)=\mathcal{C}_{\mathcal{K}}^{\rho}(z)$ for 
$z\in \setC\setminus \Omega$, 
\begin{equation*}
S(z)=\bar
z+\mathcal{C}_{\mathcal{K}}^\Omega(z)-\mathcal{C}_{\mathcal{K}}^{\rho}
(z)
\end{equation*}
is the Schwarz function of $\partial \Omega$ when $z\in\Omega$.
Similarly to the above proof,  we denote by $\psi$  the  conformal map from the
lower half-plane $\mathbb{H}_-$ onto $\Omega$ and define $F$ as in
(\ref{eq:6.23}). Obviously $F(w)$ is holomorphic in the upper half-plane
$\mathbb{H}_+$, while for $w\in \mathbb{H}_-$ we have that
\begin{equation}
\begin{split}
 \partial_{\bar w}F(w) & =\overline{\partial_{w}\psi(w)}
+\left\{
-\chi_\Omega\left(\psi(w)\right)+\rho\left(\psi(w)\right)\right\}\overline
{ \partial_{w}\psi(w)}
\\ & =\rho\left(\psi(w)\right)\overline
{ \partial_{w}\psi(w)}.
\end{split}
\end{equation} 
Hence $F$ is holomorphic 
in $\mathbb{H}_-$  apart from the preimage of ${\rm supp}(\rho)$. 
Since $S(\psi(t))=\overline{\psi(t)}=\psi^\ast(t) $ for $t\in
\setR$, $F$ has a holomorphic continuation through the real line.

We  now  set $\mu=\left(\rho\circ \psi\right)\overline{{\partial_{w}\psi}}$.
Then $\left\{\mathcal{C}^\mu(w)-
\left(\mathcal{C}^{\rho}
\left(\psi(w)\right)\right)\right\}$ is a holomorphic function in the lower
half--plane.  This means that the Cauchy transform $\mathcal{C}^\mu$ captures
the singularities of $F$.  Using inversion
as in the proof of Theorem \ref{thm:2}, we conclude that $F$ is holomorphic 
near the infinity point, and hence in 
the Riemann sphere excluding the support of $\mu$. Therfore
\begin{equation*}
 F(w)=q^\ast(w) + \mathcal{C}^\mu(w),
\end{equation*}
where $q$ is a polynomial. So by (\ref{eq:6.23}), 
\begin{equation*} 
\psi(w)=q(w)+\left(\mathcal{C}^\mu(w)\right)^\ast. 
\end{equation*} 
Applying  again  Sakai's  \cite{Sakai_91, Sakai_93} in a similar manner as we did in the previous proof, we obtain that the degree of  $q$ is at most two.

Let $K\subset\mathbb{H}_-$ denote the support of $\rho\circ\psi$, which is 
the preimage of $\supp(\rho)$. Then obviously ${\rm supp}(\mu)=K$, 
\begin{equation*}
\left(\mathcal{C}^\mu(w)\right)^\ast=\frac{1}{\pi}\overline{\int
\frac{\mu(\zeta)}{\bar w-\zeta}dA(\zeta) }=\frac{1}{\pi}\int
\frac{\overline{\mu}(\bar\zeta)}{ w-\zeta}dA(\zeta)=\mathcal{C}^{\mu^\ast}(w),
\end{equation*}
and   ${\supp}(\mu^\ast)=K^\ast\subset\mathbb{H}_+$, where $K^\ast$ is the
mirror image of $K$. 
Thus the conformal map $\psi$ is holomorphic in $\mathbb{H}_-$ and satisfies the
equation
\begin{equation}
\label{eq:2.21}
 \psi(w)=q(w)+\frac{1}{\pi}\int\frac{\mu^\ast
\left(\zeta\right)}{w-\zeta}dA(\zeta).
\end{equation} 

The boundary $\partial\Omega$ has a parametric representation 
$\{\psi(t):  t\in\setR\}$. So since $\supp(\mu^\ast)=K^\ast$  is compact,
we obtain  from (\ref{eq:2.21})  that 
\begin{equation*} 
\lim_{t\to\pm\infty}\left\{\psi(t)-q(t)\right\}=0. 
\end{equation*} 
Thus the  asymptotic curve is determined by the coefficients of  $q$  
in (\ref{eq:2.15}): A straight line if $A_2=0$; otherwise, a parabola when ${\rm Im}({A_1}/{A_2})\neq 0$ and an infinite ray when   
${\rm Im}({A_1}/{A_2})= 0$.
\end{proof}

\begin{rem}
\label{remark:1}
 Here is a simple argument showing that $\sum_{k=1}^{n+1}\gamma_k\log(w-d_k)$ 
 is single--valued outside the piecewise straight line connected the points 
$\{d_1,\ldots,d_{n+1}\}$ if and only if $\sum \gamma_k =0$.

Let $\sigma=\sum_{k=1}^{n+1}\gamma_k$. Then
\begin{equation*}
\begin{split}
 &\sum_{k=1}^{n+1}\gamma_k\log(w-d_k)\\  = & \sum_{k=1}^{n}
\gamma_k\log(w-d_k)+\left\{-(\gamma_1+\cdots+\gamma_n)+\sigma\right\}\log(w-d_{
n+1 } )\\ 
=&\gamma_1\int_{d_1}^{d_2}\dfrac{1}{w-s}
ds+(\gamma_1+\gamma_2)
\int_{d_2}^{d_3}\dfrac{1}{w-s}ds
\\& +\cdots+(\gamma_1+\cdots+\gamma_n)\int_{d_n}^{
d_{n+1}}\dfrac{1}{w-s}ds + \sigma
\log(w-d_{n+1}).
\end{split}
\end{equation*} 
Since the integrals are holomorphic  off the curve, while  $ \log(w-d_{n+1})$
is not single--valued  in any neighborhood of $d_{n+1}$, we see that 
$\sum_{k=1}^{n+1}\gamma_k\log(w-d_k)$ is 
single-valued if and only if $\sigma=0$.
\end{rem}

\section{Contact surfaces and quadrature domains}
\label{sec:contact}

We present here Strakhov's model of two dimensional contact surfaces \cite{Strakhov_74} and show its relations to unbounded QDs. 
Let $\Gamma$ be a Jordan curve  (the contact curve) in the strip $\{h_1<{\rm 
Im}(\zeta)<h_2\}$ and let  $\{\zeta(t): t\in\setR\}$ be its   parametric 
representation. We assume that  $\Gamma$ has an asymptotic line, that is, 
$\lim_{t\to\pm\infty}{\rm Im}(\zeta(t))=h$. The curve $\Gamma$  separates the 
strip into 
two layers; $D_1$ below $\Gamma$ and with a constant density $\sigma_1$, and 
$D_2$ above the curve and with a density $\sigma_2$.

Let 
\begin{equation}
U^\mu(z)=-\frac{1}{2\pi}\int \log|\zeta-z|d\mu(\zeta),
\end{equation}
be the logarithmic potential of a measure $\mu$, and whenever $\mu=\sigma\chi_D 
dA$ for some set $D$ we denote it by $U^{\sigma D}$. Note that if $D\subset 
\{h_1\leq{\rm Im}(\zeta)\leq h_2\}$, then the potential $U^{\sigma D}$ may not 
converges. However, the gravitational field in the direction perpendicular to 
the 
strip
\begin{equation*}
\partial_y U^{\sigma D}(x,y)=\frac{1}{2\pi}\int_D \sigma\frac{y-y'}{(x-x')^2+(y-y')^2}dx'dy'
\end{equation*}
always exists.  
Let $D_+$, $D_-$ be the sets in between $\Gamma$ and its asymptotic line
$\{{\rm Im}(\zeta)=h\}$, that is, $D_+=D_2\cap\{ h<{\rm Im}(\zeta)\}$ and  
 $D_-=D_1\cap\{ h>{\rm Im}(\zeta)\}$. Note that  whenever  $D$ is a strip 
$\{a\leq{\rm Im}(\zeta)\leq b\}$, then  $\partial_y U^{\sigma D}$ is a constant for ${\rm Im}(z)>b$ (see e.g. \cite{karp_margulis_96}).  Using this property we find that 
\begin{equation}
\partial_y U^{\sigma_2D_2}(z)+\partial_y U^{\sigma_1D_1}(z)=\text{const.}-\partial_y U^{\sigma D_+}(z) - \partial_y U^{-\sigma D_-}(z),
\end{equation}
when ${\rm Im}(z)>h_2$ and where $\sigma=\sigma_2-\sigma_1$. If in addition, 
$|{\rm Im}(\zeta(t))-h|=O(|t|^{-\alpha})$ for large $|t|$ and $\alpha>0$, then the 
gravitational fields $\partial_xU^{\sigma D_\pm}$ converge, and hence  
\begin{equation}
\begin{split}
2\left(\partial_xU^{\sigma D_+}(z)-i\partial_yU^{\sigma D_+}(z)
+\partial_xU^{-\sigma D_-}(z)-i\partial_yU^{-\sigma D_-}(z)\right)\\ =\frac{1}{\pi}\int_{D_+}\frac{\sigma}{\zeta-z}dA+\frac{1}{\pi}\int_{D_-}\frac{-\sigma}{\zeta-z}dA= \C^{\sigma D_+}(z)+\C^{-\sigma D_-}(z).
\end{split}
\end{equation}
Applying Green's theorem to the function $\sigma\frac{\bar{\zeta}-\zeta+2ih}{\zeta-z}$ in each component of $D_+$ and $D_-$, and noting that this function vanishes on the asymptotic line $\{{\rm Im}(\zeta)=h\}$, we get that 
\begin{equation}
\label{eq:cauchy:2}
F(z):=\frac{1}{2\pi i}\int_\Gamma \sigma\frac{\bar{\zeta}-\zeta+2ih}{\zeta-z}d\zeta =\C^{\sigma D_+}(z)+\C^{-\sigma D_-}(z), \qquad {\rm Im}(z)>h_2.
\end{equation} 
Thus up to a constant term, the Cauchy integral
determines the gravitational fields.

We shall see that the Schwarz function of $\Gamma$ governs   the complex gravitational field $F(z)$ in (\ref{eq:cauchy:2}).
Let $\Omega$ be the domain below $\Gamma$ and let $S(\zeta)$ be the Schwarz 
function of $\Gamma$ and assume that the singularities of the Schwarz function 
in $\Omega$ have compact support. Then
\begin{equation}
\label{eq:3.11}
\int_{\Omega\cap \{|\zeta|=R\}}\sigma\frac{S(\zeta)-\zeta+2ih}{\zeta-z}d\zeta+
\int_{\Gamma\cap B_R}\sigma\frac{S(\zeta)-\zeta+2ih}{\zeta-z}d\zeta
=\sigma\int_\gamma\frac{S(\zeta)}{\zeta-z}d\zeta,
\end{equation}
where $\gamma$ is a closed curve around the singularities of the Schwarz function in $\Omega$ and $R$ sufficiently large positive number.  Assume further that $\Omega$ is the image of a conformal mapping $\psi(w)=w +ih +\varphi(w)$ in the lower half-plane, then 
\begin{equation*}
S(\zeta)=\psi^{-1}(\zeta)-ih+\varphi^\ast\left(\psi^{-1}(\zeta)\right)
\end{equation*}
is the Schwarz function of $\Gamma$. Therefore, if $\varphi$ is given by 
(\ref{eq:2.14}), then $S(\zeta)-\zeta+2ih=o(|\zeta|)$, as $|\zeta|\to \infty$, 
and this implies that the first integral of the left hand side of 
(\ref{eq:3.11}) tends to zero as $R\to \infty$. So we conclude that 
\begin{equation*}
F(z)=\frac{1}{2\pi i}\int_\Gamma \sigma\frac{{\bar\zeta}-\zeta+2ih}{\zeta-z}d\zeta=\frac{1}{2\pi i}\int_\Gamma \sigma\frac{S({\zeta})-\zeta+2ih}{\zeta-z}d\zeta =\frac{\sigma}{2\pi i}\int_\gamma\frac{S(\zeta)}{\zeta-z}d\zeta.
\end{equation*} 
On the other hand, by (\ref{eq:2.7}),
\begin{equation*}
\int_\Omega f dA=\frac{1}{2i}\int_\Gamma S(z)f(z)dz=\frac{1}{2i}\int_\gamma\frac{S(\zeta)}{\zeta-z}d\zeta.
\end{equation*}
Thus the singularities of the Schwarz function control the complex potential $F(z)$ as well as the distribution $T$ of the quadrature identity (\ref{eq:q-i}). The Cauchy integral (\ref{eq:cauchy:2}) of the curve $\Gamma$ is a rational function if and only if the domain below it is a QD of a combination Dirac measures and their derivatives.

\section{Examples of non--uniqueness}
\label{sec:example}

In this section we use the specific form of the conformal map (\ref{eq:6.12}) 
and present examples of a continuous family of domains such that each one of them is a QD of the same measure. In Examples \ref{ex:1} and \ref{ex2} the families 
will converge to an union of a disk and a null QD. 

The idea of these constructions is due to Strakhov \cite{Strakhov_74_2}, where 
he used the conformal mapping  (\ref{eq:6.12}) with $q$ a linear polynomial (see also \cite{fedorova_tsirulskiy}). In a particulate case where the perturbation $\varphi$ in (\ref{eq:6.12}) has a single simple pole, 
Strakhov  computed the boundary explicitly and showed that it is  a third order algebraic curve. For the convenience of the reader we present
 his example here.

\begin{exmp}[Strakhov ]
\label{ex:1}

We consider a conformal map of the form (\ref{eq:6.12}), 
\begin{equation*}
\label{eq:3.1}
 z=\psi(w) = w +ih +\frac{a}{w-ib}, \quad a, b>0,\  h\in\setR.
\end{equation*} 
The parameters $a$, $b$ and $h$ are chosen  so that $\psi(-ib)=0$, and that the
residue  of the Schwarz function at zero is  one. From equations (\ref{eq:2.7}) and (\ref{eq:2.9}), and the above expression for $\psi$ we have that 
\begin{equation}
\label{eq:3.2b}
 \left\{\begin{array}{l}
h=b-\frac{a}{2b}\\  a+\frac{a^2}{4b^2}=1        
   \end{array}\right..
\end{equation}
Since there are three parameters, one of them is free. Thus for an appropriate  
choice of the parameters $a, b$ and $h$,  there is a one parameter family of 
conformal mappings, which we will denote by  $\psi_b$, and such that $\Omega_{b}:=\psi_b(\mathbb{H}_-)$ becomes  
a QD for $\pi\delta_0$, where $\delta_0$ is the Dirac measure at $0$.

The boundary of $\Omega_{b}$ is given by 
\begin{equation*}
 \Gamma_{b}=\left\{z(t)=t+ \frac{at}{t^2+b^2} + i\left(h+
\frac{ab}{t^2+b^2}\right): t\in\setR\right\}.     
\end{equation*}
Using polar coordinates one may show that $\Gamma_{b}$ has the implicit
representation
\begin{equation}
\label{eq:3.2}
\{(y+r-\alpha)\left(x^2+(y+r)^2\right)=2r(y+r)^2,\quad \alpha,r>0\},
\end{equation} 
where $\alpha=b$, $r=\frac{a}{2b}$ and $h=\alpha-r$. This   is a
part of a larger family of third order curves that  are  called {\it Conchoids
of de Sluze}. From (\ref{eq:3.2b}) we see that 
$\alpha=\frac{1}{2}(\frac{1}{r}-r)$. Thus
 $\alpha\to 0$ when $r\to 1$, and from the implicit representation 
(\ref{eq:3.2}) it is clear that $\Gamma_{b}$ converges to the union of the
circle $\{x^2+y^2=1\}$  and the line $\{y=-1\}$. Therefore, as $(a,b)\to(0,0)$
and (\ref{eq:3.1}) is 
fulfilled, then $\Omega_{b}$ converges to the union of the unit disk and the
null QD  $\{y<-1\}$.

\begin{figure}[h]
\centering
\includegraphics[width=0.7\linewidth]{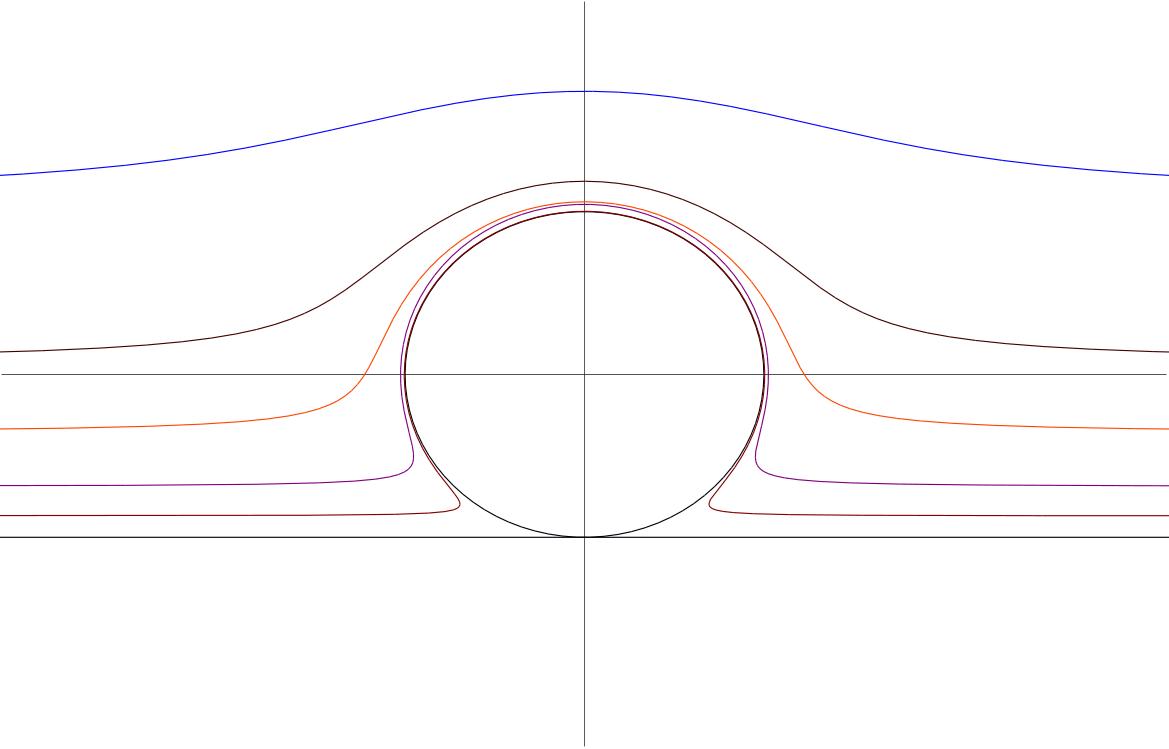}
\caption{A family of Conchoids of de Sluze}
\label{fig:conchoid6}
\end{figure}

\end{exmp}

\begin{rem}
In a similar manner we can fix the nodes and the coefficients of the quadrature
identity and construct  a rational conformal mapping of the form (\ref{eq:6.12})
--(\ref{eq:2.14}). This will result in a system of algebraic equations that has 
one  free parameter, and hence provides  examples of families of domains such that each one of them is a QD of the same distribution. 
\end{rem}

\begin{exmp}
\label{ex2}
Following Strakhov's example we   construct a family of 
QDs where a parabola is the  asymptote of the boundary. We consider a map 
\begin{equation}
 z=\psi(w)= 2w+iw^2+ih+\frac{a}{w-ib}, \quad w\in \mathbb{H}_-,\ a,b>0,\  h\in\setR. 
\end{equation} 

Assuming for a moment that the map $\psi$ is univalent, and requiring that    
$\psi(-ib)=0$ and that the Schwarz function has 
residue one at the origin, then we get the following two equations
\begin{equation}
\label{eq:3.6b}
 \left\{\begin{array}{l}
  h=2b+b^2-\frac{a}{2b}\\   2a+2ab +\frac{a^2}{4b^2}=1
  \end{array}\right..
\end{equation} 
Since the algebraic equations (\ref{eq:3.6b}) has one parameter free, there is a family of conformal mappings $\psi_a$ such that $\Omega_{a}=\psi_a(\mathbb{H}_-)$ is a QD of the unit Dirac measure. 
Our aim is to let $a,b$ tend to zero in a such manner that 
$\Omega_{a}$ will converge. 
From the second equation of (\ref{eq:3.6b}) we see that the condition 
\begin{equation}
\label{eq:3.6}
 \lim_{a,b\to0}\frac{a}{b}=2
\end{equation} 
is demanded. 

In order to assure that the map $\psi_a$ is univalent, we will show that for an
appropriate  choice of small parameters $a$ and $b$ it 
maps the real line onto a curve without closed loops.  To see this we let
$z(t)=X(t)+iY(t)$, where
\begin{equation}
\label{eq:3.4}
 X(t)= 2t + \frac{at}{t^2+b^2}, \qquad Y(t)=t^2+h+\frac{ab}{t^2+b^2}, 
\end{equation} 
 be the parametric presentation  of the boundary of $\Omega_a$, 
and we shall and analyze the critical points of these functions.  
Furthermore, since   $X$ is an odd function and  $Y$ 
is an even, it suffices to examine the critical points only for 
positive $t$. 
Computing the derivatives
\begin{equation}
\label{eq:3.5}
 \frac{dX}{dt}(t)=\frac{2(t^2+b^2)^2+a(b^2-t^2)}{(t^2+b^2)^2}, \qquad 
 \frac{dY}{dt}(t)=2t-\frac{2abt}{(t^2+b^2)^2},
\end{equation} 
we find that the function $Y$ has critical points when
$t^2=b\sqrt{\frac{a}{b}}-b^2$ and
$b\sqrt{\frac{a}{b}}-b^2\geq0$. 
 The critical points of $X$ satisfy the equation
$t^2=\frac{1}{4}\left(\pm\sqrt{a^2-16ab^2}+a-4b^2\right)$.  This equation has   two positive roots when $a$ and $b$ are small. Now,  the contour $z(t)=X(t)+iY(t)$ will have a closed loop for positive $t$ if
and only if the critical point of  $Y$ is in between the two critical points of 
$X$. The largest root is approximately $\frac{1}{4}\left(2a-12b^2\right)$ , and  
since by (\ref{eq:3.5})  $a\simeq 2b$,   it is less than $b\sqrt{\frac{a}{b}}-b^2 
$ for small $b$. So we conclude that the curve $z(t)$ has no closed loops when 
$a$ and $b$ are sufficiently small.

Having showed that $\psi_a$ is univalent, we turn now to compute the limit of the boundary of $\Omega_{a}$, $\Gamma_{a}=\psi_a(\setR)$, as $a$ and $b$ tend to 
zero. 
However, unlike Example \ref{ex:1}, we do not know the explicit representation 
of the boundary. Therefore we make the variable change
\begin{equation*}
 \cos\theta=\frac{1}{\sqrt{(t/b)^2+1}}, 
\qquad \sin\theta=\frac{t/b}{\sqrt{(t/b)^2+1}}, \quad 
-\frac{\pi}{2}<\theta<\frac{\pi}{2}.
\end{equation*}
 Then (\ref{eq:3.4}) becomes
\begin{equation}
\label{eq:3.7}
 X(\theta)= 2b\tan\theta + \frac{a}{b}\sin\theta\cos\theta, \qquad 
Y(\theta)=b^2\tan^2\theta+h + \frac{a}{b}\cos^2\theta.
\end{equation} 
By (\ref{eq:3.6b}) and (\ref{eq:3.6}),
 $\lim_{a,b\to0} h=-1$. Therefore, for any
$\epsilon>0$ and $\theta\in [-\frac{\pi}{2}+\epsilon,\frac{\pi}{2}-\epsilon]$, 
the curve (\ref{eq:3.7}) tends to 
\begin{equation*}
 X(\theta)= 2\sin\theta\cos\theta = \sin 2\theta, \qquad 
Y(\theta)=2\cos^2\theta-1=\cos 2\theta.
\end{equation*} 
On the other hand, we see that for $t^2\geq b$ 
\begin{equation*} 
|X(t)-2t|=\frac{a|t|}{t^2+b^2}=\frac{a}{|t|}\frac{1}{1+\left(\frac{b}{t}
\right)^2}\leq \left(\frac{a}{b}\right)\sqrt{b}\to 0
\end{equation*}
and 
\begin{equation*}
 |Y(t)-(t^2-1)|\leq |h+1|+\frac{a}{1+b}\to 0,
\end{equation*}
as $(a,b)\to (0,0)$ and equations (\ref{eq:3.6b}) hold. 
Thus the family of curves $\Gamma_{a}$ tends  to a union of the unit circle and the parabola $y+1=(x/2)^2$. The family of simply connected QDs converges to 
\begin{math} 
B_1(0)\cup\left\{y<\left({x}/{2}\right)^2-1\right\},
\end{math}
that is, a union of null QD and a disk. 

\begin{figure}[h]
\centering
\includegraphics[width=0.4\linewidth, height=0.2\textheight]{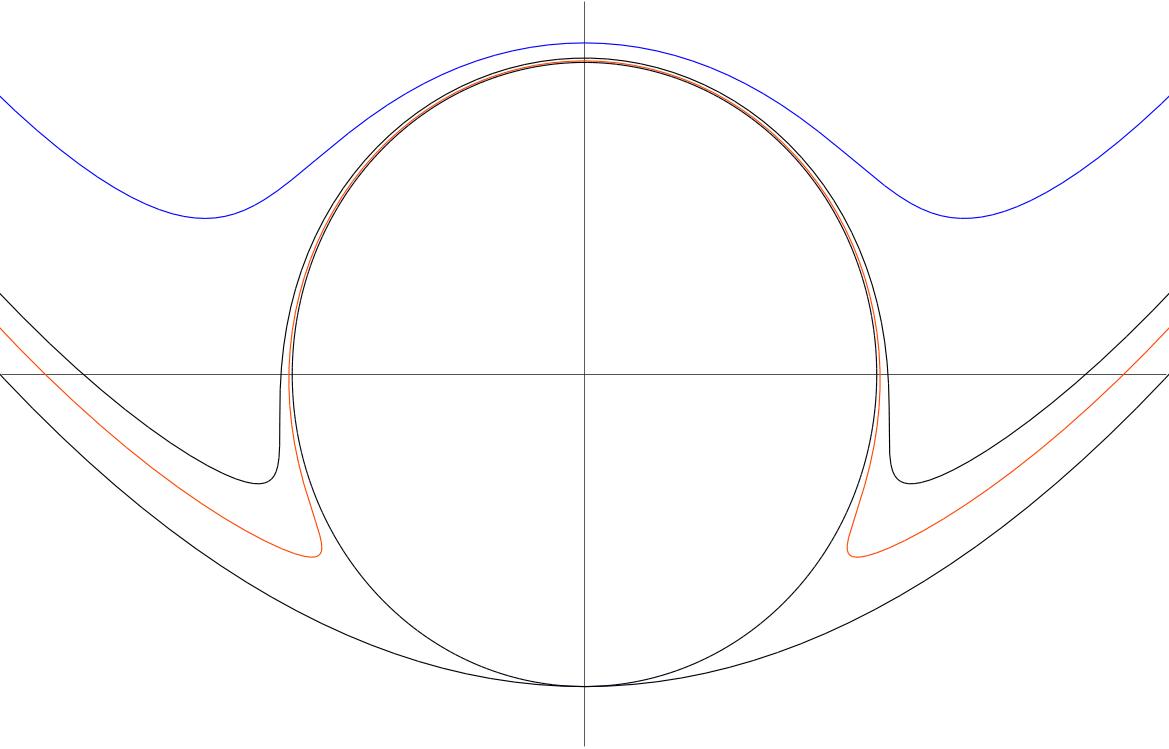}
\caption{A family converging to a parabola and a circle}
\label{fig:parabola4}
\end{figure}

\end{exmp}

The existence of a QD for a positive measure and  with  an infinite ray as an 
asymptote of the boundary is not evident. It cannot be constructed by Sakai's 
variational method \cite[Ch. 11]{Sakai_82}, since this method requires that the 
complement of the attached null QD has  non--empty interior.  Nevertheless, 
 by using similar ideas as in examples \ref{ex:1} and \ref{ex2} we are able to  
construct a family of QDs for a positive  Dirac measure at zero and with the 
positive $x$--axis as the asymptote of the boundary. In  contrast with the 
previous two examples, the boundary of the family cannot converge to a union of 
a circle and a ray, since this contradicts the regularity of the Schwarz 
function \cite{Sakai_91}.

\begin{exmp}
We consider a conformal map 
 \begin{equation}
\label{eq:3.8}
  z=\psi(w)=w^2+h-\frac{ia}{w-ib}, \quad b>0, \ a,h\in\setR
 \end{equation}
from $\mathbb{H}_-$ to the domain $\Omega$. 
Then the requirements that $\psi(-ib)=0$, the Schwarz function has  residue
one at $z=0$, and the origin belongs to the image of $\psi$, leads to the 
following relations:
\begin{subequations}
 \label{eq:3.9}
\begin{align}
   \label{eq:3.9a}
  & h =b^2-\frac{a}{2b}& \\
    \label{eq:3.9b} &8b^3a-4b^2+a^2=0&\\
&h+\frac{a}{b}=b^2+\frac{a}{2b}>0.&
\end{align}
\end{subequations}

For $a<0$ relations (\ref{eq:3.9}) cannot hold.
When  $a>0$, we find by computing the minimum of the third degree
polynomial in (\ref{eq:3.9b})  that it has positive roots only when
$a\leq\sqrt[4]{4/27}$. For that range of $a$ there are two types of QDs. 

For both types we need  to check that the map (\ref{eq:3.8}) is
conformal. We do this by checking the  conditions which  guarantee that  the
real
line is mapped in a one to one manner onto the curve
\begin{equation}
\label{eq:3.10}
 X(t)= t^2+h + \frac{ab}{t^2+b^2}, \qquad Y(t)=\frac{-at}{t^2+b^2}, \ t\in\setR.
\end{equation} 

The first type is when the function $X$ has critical points for $t\neq0$, this occurs  when $t^2=\sqrt{ab}-b^2$, which implies that $b^3<a$. 
Then the curve (\ref{eq:3.10}) will not have
a closed loop if  the critical points of $Y$ appear ``after'' the critical
points of $X$, which means that $b^2\geq \sqrt{ab}-b^2$. Thus in that case we 
have to
require that  the largest root of (\ref{eq:3.9b}) will 
satisfy the condition
$b^3<a<4b^3$. 

The second type is where  the function $X$ is monotone for $t\gtrless0$.
Also, in that
case  that the largest root of (\ref{eq:3.9b}) needs to satisfy the 
condition $b^3>a$.  Note that if $a<1/3$, then the largest root is greater than
one and hence this  condition is satisfied.  

\begin{figure}[h]
\centering
\includegraphics[width=0.5\linewidth, height=0.3\textheight]{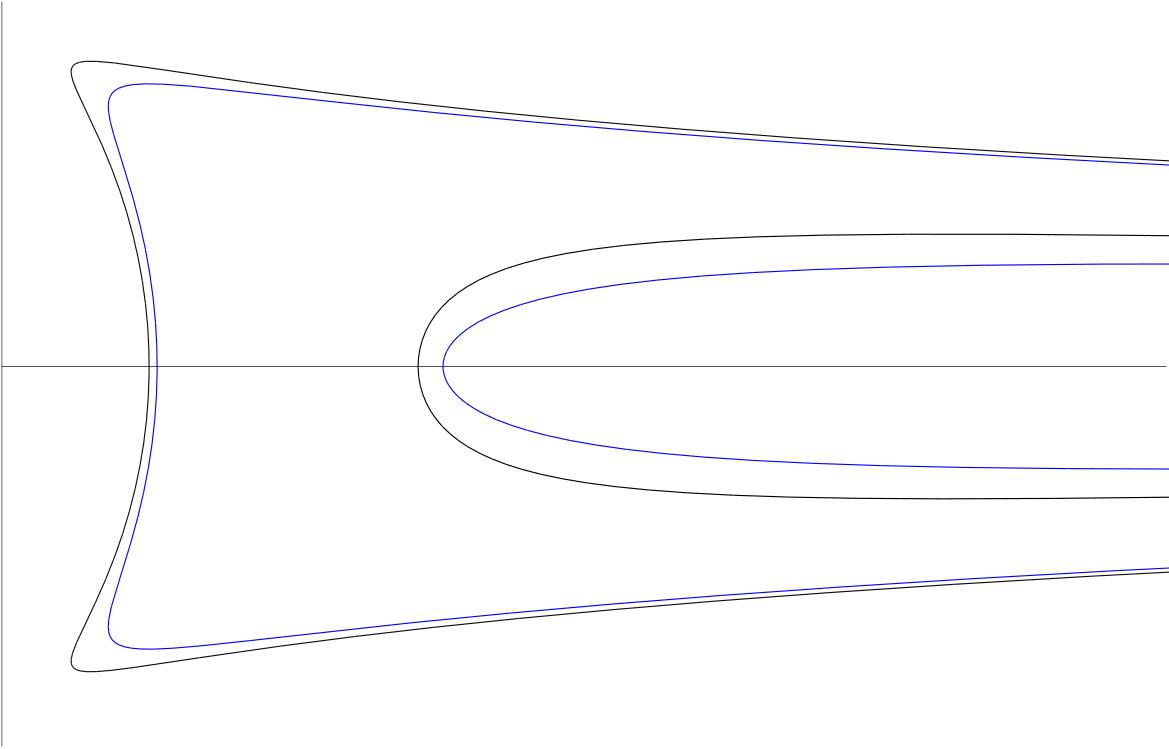}
\caption{Two types of QDs  for the Dirac measure and a ray as an asymptote}
\label{fig:ray5}
\end{figure}


\end{exmp}

\section{Summary}
By means of conformal mappings we established the 
asymptotic behavior of the boundary of unbounded quadrature domains in the plane 
and when the infinity point belongs to the boundary. Although this tool  is 
not available in higher dimensions, we hope the present paper  will stimulate 
 further investigations of unbounded quadrature domains in the space. The 
specific form of the conformal mapping from the lower half plane enables 
the construction of families of quadrature domains of the Dirac measure at a
given point and possessing a given type of the asymptote.     


\vskip 10mm
\noindent
\textbf{Acknowledgement.}
I would like to thanks Avmir Margulis for many valuable talks and enlightening 
comments. I also grateful  to the anonymous  referee for his/her constructive 
comments, which definitely contributed  to the  improvement of the manuscript.

\bibliographystyle{amsplain}
\bibliography{bibgraf-quad}

\end{document}